\documentclass[12pt]{amsart}
 \newtheorem{thm}{Theorem}[section]

 \theoremstyle{definition}
 \newtheorem{defn}[thm]{Definition}
 \theoremstyle{remark}
 \newtheorem{rem}[thm]{Remark}
 \newtheorem{ex}[thm]{Example}
 \numberwithin{equation}{section}
 \newcommand{\eps}{\varepsilon}

 \newcommand{\Real}{\mathbf{R}}

  \newcommand{\dd}[1]{\frac{d}{d #1}}

 \newcommand{\Hess}{\mathrm{Hess}}
 \newcommand{\Azero}{{\bf (A0)}}
 \newcommand{\Aone}{{\bf (A1)}}
 \newcommand{\Atwo}{{\bf (A2)}}
 \newcommand{\Athreew}{{\bf (A3w)}}
 \newcommand{\Athrees}{{\bf (A3s)}}
 \newcommand{\CCw}{{\bf (B3w)}}
 \newcommand{\CCs}{{\bf (B3s)}}
\begin{document}

\title[MTW curvature for natural mechanical actions]
{%Jacobi Fields and
The Ma-Trudinger-Wang curvature for natural mechanical actions}

\author{Paul W.Y. Lee}
\email{plee@math.toronto.edu}
\address{Department of Mathematics, University of Toronto, ON M5S 2E4, Canada}

\author{Robert J. McCann}
\email{mccann@math.toronto.edu}
\address{Department of Mathematics, University of Toronto, ON M5S 2E4, Canada}
\date{\today}

\thanks{The first author's research was partially supported by the NSERC postdoctoral fellowship and the second author's research was supported by NSERC grant 2170006-08. \copyright 2009 by the authors}

\begin{abstract}
The Ma-Trudinger-Wang curvature --- or cross-curvature --- is an object arising in the regularity theory
of optimal transportation.
If the transportation cost is derived from a Hamiltonian action,  we show its
cross-curvature can be expressed
%reexpress this object
in terms of the associated Jacobi fields.  Using this expression,
we show the least action corresponding to a harmonic oscillator has zero cross-curvature, and
in particular satisfies the necessary and sufficient condition \Athreew\ for
the continuity of optimal maps.
We go on to study gentle perturbations of the free action by a potential, and deduce conditions on the
potential which guarantee either that the corresponding cost satisfies the more restrictive
condition \Athrees\ of Ma, Trudinger and Wang, or in some cases has positive cross-curvature. In particular,
the quartic potential of the anharmonic oscillator satisfies \Athrees\ in the perturbative regime.
\end{abstract}

\maketitle

\section{Introduction}

Let $\mu$ and $\nu$ be two Borel probability measures on the manifold $M$ and let $c:M\times M\longrightarrow\Real$ be a cost function. The theory of optimal transportation begins with the following minimization problem:

Find a Borel map which minimizes the following total cost among all Borel maps $\varphi:M\longrightarrow M$ which push $\mu$ forward to $\nu$:
\begin{equation}\label{Monge}
\int_Mc(x,\varphi(x))d\mu(x).
\end{equation}
Here the push forward $\varphi_{*}\mu$ of a measure $\mu$ by a Borel map $\varphi$ is the measure defined by $\varphi_{*}\mu(U)=\mu(\varphi^{-1}(U))$ for all Borel sets $U\subseteq M$.

Recently, there have been a series of breakthroughs in understanding regularity of the solution
to this optimal transportation problem. After the work of \cite{MaTrWa,TrWa,Lo1,KiMc1}, it is clear that
regularity of the optimal map is dictated by a quantity introduced by Ma, Trudinger and Wang,
called the Ma-Trudinger-Wang (MTW) curvature \cite{FiRiVi2},
the cost-sectional curvature \cite{Lo1}, or the cross-curvature \cite{KiMc1}
(see Section \ref{background} for the definitions and conventions used in this paper).
Efforts have been directed at understanding the corresponding condition for smoothness
and finding examples which satisfy it. Still, little is known about this
curvature and much of the work has been focused on cost given by the square of a
Riemannian distance \cite{Lo2} \cite{KimCounterexample} \cite{FiRi} \cite{FiRiVi1}.

In this paper, we study the cross-curvature for costs arising from natural mechanical systems. More precisely, let $\left<\cdot,\cdot\right>$ be a Riemannian metric on the manifold $M$ and let $|\cdot|$ be the corresponding norm. Let  $V:M\longrightarrow \Real$ be a smooth function and consider cost functions $c$ defined by minimizing the corresponding mechanical action. Defining the Lagrangian $L(x,v)=\frac{1}{2}|v|^2-V(x)$ we let
\begin{equation}\label{cost}
c_T(x,y)=\inf\int_0^TL(\gamma(t),\dot\gamma(t))dt,
\end{equation}
where the infimum is taken over all smooth curves $\gamma(\cdot)$ satisfying $\gamma(0)=x$ and $\gamma(1)=y$. We recall here that the problem of finding curves which achieve the infimum in (\ref{cost}) is called Hamilton's principle of least action. We also call the minimizers to the above infimum
{\em curves of (or paths of) least action}.

In Theorem \ref{cross} below, we give a characterization of the cross-curvature in terms of Jacobi fields.  We use it to find a number of examples of costs which are non-negatively
cross-curved.
%As a consequence, we find a family of new costs which is non-negatively cross-curved.
The first of these is the family of harmonic oscillator actions.
In fact, the cross-curvature vanishes completely for this new family,
which includes the Euclidean distance squared cost in case $A=0$.

\begin{thm}[Cross-curvature vanishes for the harmonic oscillator]\label{eg}
Let $L:T\Real^n\longrightarrow\Real$ be the Lagrangian defined by
\[
L(x,v)=\frac{1}{2}|v|^2-\frac{1}{2}x\cdot Ax,
\]
where $A$ is any symmetric non-positive definite matrix. Then the cross-curvature of the corresponding cost defined by (\ref{cost}) is identically zero.
\end{thm}

Non-positive definiteness of the matrix $A$ above is only needed for boundedness of the Lagrangian $L$. One can define a potential $V$ which is equal to $\frac{1}{2}x\cdot Ax$ on a large ball $B$ centered at the origin and stay bounded outside $B$. The cross-curvature of the corresponding cost will be identically zero in a smaller ball contained in $B$.

%Although Theorem \ref{eg} was derived and discovered as a consequence of Theorem \ref{cross},
A posteriori, Theorem \ref{eg} can alternately be verified from the explicit form of the induced action. For instance when $A$ is negative definite with eigenvalues $-\lambda_1^2,...,-\lambda_n^2$, the transportation cost is given by 
$$
c_T(x,y) = \sum_{i=1}^n\frac{\lambda_i}{2\sinh(\lambda_iT)}\left[(x_i^2+y_i^2)\cosh\lambda_iT-2x_iy_i\right]. 
$$
Since this action differs from a coordinate reparameterization of the standard cost
$\tilde c(x,y) = - x \cdot y$ by null Lagragians  --- which depend separately on $x$ and on $y$,
hence do not affect the minimization \eqref{Monge} ---
the vanishing of all cross-curvatures follows immediately from the corresponding result
for $\tilde c$ and the coordinate independence described in
\cite{Lo1} \cite{KiMc1} \cite{Tr}.  However,  we originally discovered this vanishing
as a simple application of Theorem \ref{cross}; see Section \ref{S:HO} below.  Absent an
explicit evaluation of the action integral --- which is not possible in more complicated settings
--- it is difficult to guess the cross-curvature of a cost function without such tools as our theorem provides.

The importance of the example from Theorem \ref{eg} is the following. In the Riemannian case, it is known that non-negative sectional curvature is a necessary condition for the cost to satisfy the weak MTW condition or to be non-negatively cross curved (see \cite{Lo1}). However, the curvature term of the Jacobi field equation \eqref{Jacobi equation} for the above example is given by $A$ which is non-positive definite. This shows that the connection between the MTW condition and the non-negativity of the sectional curvature in the Riemannian case is atypical.

We also consider the perturbed Lagrangian $L_\eps(v)=\frac{1}{2}|v|^2-\eps V(x)$ on the trivial tangent bundle $T\Real^n$. We find computable conditions on $V$ for which the corresponding cost defined by (\ref{cost}) satisfies the strong MTW condition for all small enough $\eps >0$. More precisely,

\begin{thm}[Perturbed actions which become \Athrees]\label{goodperturb}
Assume that there exists a constant $C>0$ such that
\[
\int_0^1\int_0^\tau\left<u,(1-t)\partial_s^2 \Hess\, V_{x+t(v+sw)}u\right>\Big|_{s=0}dtd\tau\geq C
\]
for all $(x,v)$ in a bounded open subset of the tangent bundle $T\Real^n$ and for all unit tangent
vectors $u,w$ in the tangent space $T_x\Real^n$ which are orthogonal to each other.
Then there is an open set in the product $\Real^n\times\Real^n$ on which the cost $c$ defined by $L_\eps$ and (\ref{cost}) satisfies the strong MTW condition \Athrees\ for all small enough $\eps>0$.
%%Have problems putting positive cross-curvature here without putting most of the background material in section 2 into the introduction%%
\end{thm}

Theorem \ref{goodperturb} tells us that the MTW condition is more related to change in the
curvature along c-segments than to the curvature itself. Note also that the condition in
Theorem \ref{goodperturb} is easily computable whereas the cross-curvature is hard to compute
in general since --- although it is local in the product manifold $M\times M$ \cite{KiMc1}---
it is nonlocal in $M$.
We further illustrate this computability by considering the radially symmetric potentials,
i.e.\ $V(x)=f\big(\frac{|x|^2}{2}\big)$. In this case, the condition in
Theorem \ref{goodperturb} is satisfied if
\begin{equation}\label{simplegoodperturb}
f''\big({\textstyle \frac{|x+tv|^2}2}\big)\geq C>0,
\quad f'''\big({\textstyle \frac{|x+tv|^2}2}\big)\geq 0,
\quad f^{(4)}\big({\textstyle \frac{|x+tv|^2}2}\big)\geq 0.
\end{equation}

In particular, the conditions in (\ref{simplegoodperturb}) are satisfied if $f(z)=z^2$,
which arises from the physical model of the anharmonic oscillator.

\section{Background:  Cross-curvature and MTW Conditions}\label{background}

In this section, we will review some basic facts about the optimal transportation problem needed in this paper. The assumptions in the theorems stated in this section are simplified to avoid heavy notation. The corresponding theorems with relaxed assumptions can be found, for instance, in \cite{Vi}.

Let $M$ and $N$ be two smooth manifolds (possibly with boundaries) and let $\mu$ and $\nu$ be Borel probability measures on $M$ and $N$, respectively, with compact support. Let $c:M\times N\longrightarrow\Real$ be a bounded continuous cost function, so that the corresponding optimal transportation problem is the following:

Find Borel maps which minimize the following functional among all Borel maps $\varphi:M\longrightarrow N$ which push $\mu$ forward to $\nu$ (ie. $\mu(\varphi^{-1}(U))=\nu(U)$ for all Borel sets $U \subset N$):
\[
\int_Mc(x,\varphi(x))d\mu(x).
\]

The sufficient conditions for the above problem to have a unique solution are given by the following theorem (see \cite{Vi} for the proof).

\begin{thm}[Existence and uniqueness of optimal maps]\label{MongeExist}
Suppose that the cost $c$ and the measure $\mu$ satisfy the following assumptions
\begin{enumerate}
  \item $\mu$ is absolutely continuous with respect to the Lebesgue measure,
  \item the cost $c(x,y)$ is locally lipschitz in $x$, uniformly in $y$,
  \item $c$ is superdifferentiable everywhere,
  \item the map $y\longmapsto d_x c(x,y)$ is injective on its domain of definition.
\end{enumerate}
Then there is a solution $\varphi$ (called the optimal map) to the above optimal transportation problem. Moreover, it is unique $\mu$-almost everywhere.
\end{thm}

In order to discuss regularity of the optimal map, assumptions stronger than those stated in Theorem \ref{MongeExist} are needed. First, we need an open set $\mathcal O=\mathcal M\times\mathcal N$ contained in the product $M\times N$ for which the following assumptions hold:
\begin{description}
    \item[\Azero] $c$ is $C^4$ smooth on $\mathcal O$,
    \item[\Aone] the map $y\longmapsto d_xc(x,y)$ is injective on $\mathcal N$ for each $x$ in $\mathcal M$,
    \item[\Atwo] the map $v\longmapsto d_x d_y c(x,y)(v)$ from the tangent space $T_yN$ to the cotangent space $T_x^*M$ is injective for all $(x,y)$ in $\mathcal O$.
        %%Used d rather D for differential%%
\end{description}

In this paper, we focus on transportation costs given by (\ref{cost}). The above conditions $\Azero, \Aone, \Atwo$ corresponding to these costs are discussed in the Appendix. 

Next, we discuss the most important condition in the regularity theory of optimal maps;
called the Ma-Trudinger-Wang (MTW) condition, it involves the cross-curvature.
To do this, let $\mathcal K_x$ be the subset of the cotangent space $T^*_xM$ defined by
\[
\mathcal K_x=\{(x,-d_xc(x,y))|y\in\mathcal N\}
\]
and let $\mathcal K=\bigcup\limits_{x\in\mathcal M}\mathcal K_x$.

\begin{defn}[$c$-exponential]
The c-exponential map $\exp^c:\mathcal K\longrightarrow M$ corresponding to the cost function $c$ is defined by
\begin{equation}\label{c-exponential}
\exp^c(x,\alpha)=[d_x c(x,\cdot)]^{-1}(-\alpha).
\end{equation}
\end{defn}

\begin{defn}[Cross-curvature]\label{crosscurvature}
The cross-curvature $\mathcal C:TM\oplus\mathcal K\oplus T^*M\longrightarrow\Real$ corresponding to the cost $c$ is defined by
\[
\mathcal C_x(u,\alpha,\alpha_1)=-\frac{3}{2}
\partial_s^2\partial_t^2c(\gamma(t),\exp^c(x,\alpha+s \alpha_1))\Big|_{t=s=0}
\]
where $\gamma(\cdot)$ is any curve which satisfies $\gamma(0)=x$, $\dot\gamma(0)=u$. Curves of the form $s\longmapsto \exp^c(x,\alpha+s \alpha_1)$ in the definition of the cross-curvature are called c-segments.
\end{defn}

In \cite{KiMc1}, it is shown that the cross-curvature defined above can be characterized as the
sectional curvature of a certain semi-riemannian structure on $\mathcal O$ in which $c$-segments
form lightlike geodesics --- and hence the name.
The same quantity has also been called the Ma-Trudinger-Wang curvature \cite{FiRiVi2}
since it appeared first in \cite{MaTrWa}. Let $\mathcal S$ be the set defined by
\[
\mathcal S_x:=\{(x,u,\alpha,\alpha_1)\in T_xM\oplus\mathcal K_x\oplus T_x^*M|\alpha_1(u)=0\}.
\]
and let $\mathcal S=\bigcup\limits_{x\in\mathcal M}\mathcal S_x$. In this paper, we use the name Ma-Trudinger-Wang curvature to denote the restriction of the cross-curvature $\mathcal C$ to the set $\mathcal S$.

\begin{defn}\label{MTWcurvature}
The Ma-Trudinger-Wang curvature $MTW$ is defined by
\[
MTW=\mathcal C\Big|_{\mathcal S}.
\]
\end{defn}

Finally, we can state the MTW conditions and the cross-curvature conditions.

\begin{description}
    \item[\Athreew] the cost $c$ satisfies the weak MTW condition on $\mathcal O$ if $MTW\geq 0$ on $\mathcal S$,
    \item[\Athrees] the cost $c$ satisfies the strong MTW condition on $\mathcal O$ if it satisfies the weak MTW condition on $\mathcal O$ and $MTW_x(u,\alpha,\alpha_1)=0$ only if $u=0$ or $\alpha_1=0$,
    \item[\CCw] the cost $c$ is non-negatively cross-curved on $\mathcal O$ if $\mathcal C\geq 0$,
    \item[\CCs] the cost $c$ is positively cross-curved on $\mathcal O$ if it is non-negative cross-curved on $\mathcal O$ and  $\mathcal C_x(u,\alpha,\alpha_1)=0$ only if $u=0$ or $\alpha_1=0$.
\end{description}

The relevance of these conditions to the regularity theory of optimal maps can be found
in \cite{MaTrWa,Lo1,Lo2,KiMc1,TrWa,LoVi,FiLo,FiRi,FiRiVi1,FiRiVi2,FiKiMc2}. Recently,
it was shown  that the conditions \CCw\ and \CCs\ also have interesting
applications to microeconomics \cite{FiKiMc1} and statistics \cite{Sei09p} which
have little to do with smoothness.

\section{The MTW Curvature and the Jacobi Map}

Let $\mathcal U$ be an open subset of the cotangent bundle $T^*M$ such that all elements in $\mathcal U$ are regular and not on the conjugate locus (the definitions of a regular covector and the conjugate locus are analogous to the Riemannian case, see appendix for the precise definitions). 

For the rest of the paper, we consider Lagrangians $L:TM\to\Real$ which are of the form $L(x,v)=\frac{1}{2}|v|^2-V(x)$, where $|\cdot|$ denotes the norm corresponding to a fixed Riemannian metric. We also identify the tangent and the cotangent bundle by this Riemannian metric. 

Let $u,v,w$ be tangent vectors contained in the tangent space $T_xM$ at a point $x$ and assume that $v$ is contained in the open set $\mathcal U$. Let $\tau \in \Real$ parameterize a family of least action paths $t \in [0,1] \longmapsto\varphi(t,\tau)$ corresponding to the cost (\ref{cost}) with initial velocity $v+\tau w$. Let $J:=\partial_\tau\varphi\Big|_{\tau=0}$ be the corresponding Jacobi field. Let $D_t$ be the covariant derivative along the curve $t\longmapsto\varphi(t,\tau)$.
Theorem \ref{potentialJacobi} asserts
the Jacobi field $J$ satisfies the well-known equation
\begin{equation}\label{Jacobi equation}
D^2_t J+R(\partial_t\varphi,J)\partial_t\varphi+\Hess\,V_{\varphi}(J)=0.
\end{equation}

All Jacobi fields are solutions to the above second order differential equation, the Jacobi equation.
We can associate two problems to this differential equation: the initial value problem and the
boundary value problem. We use the term Jacobi map to refer to the map which sends the
initial value of the boundary value problem to the initial derivative
\eqref{Jacobi map} of the corresponding solution whose terminal value is (its first) zero.
More precisely, let $x$ and $y$ be two points on the manifold $M$ which can be joined by a
unique path $\gamma(\cdot)$ of least action.
Let $t\longmapsto J(t)$ be the Jacobi field along $\gamma(\cdot)$ such that $J(0)=u$, $J(1)=0$, and $J(t)\neq 0$ for all $t$ in the interval $(0,1)$. We define the {\em Jacobi map} $\mathcal J^c: M\times TM\longrightarrow TM$ by
\begin{equation}\label{Jacobi map}
 \mathcal J^c(y,u)=\dd t J\Big|_{t=0}.
\end{equation}
It is not hard to see that the Jacobi map is linear in the variable $u$.

The computation of the cross-curvature boils down to the computation of the Jacobi map according to the following theorem.
In the special case of a {\em Riemannian} action, Figalli, Rifford and 
Villani developed a related result independently, see 
Proposition 2.4 of \cite{FiRiVi1},  as we learned after 
the original draft of the present manuscript was complete.

\begin{thm}[Cross-curvature and the Jacobi map]\label{cross}
 The cross-curvature $\mathcal C$ is given in terms of the Jacobi map $\mathcal J^c$ by
\[
\mathcal C(u,v,w)=\frac{3}{2}\partial_s^2\left<u,\mathcal J^c(\exp^c(v+sw),u)\right>\Bigg|_{s=0}.
\]
Here $\exp^c$ denotes the c-exponential \eqref{c-exponential}.
\end{thm}

\begin{proof}[Proof of Theorem \ref{cross}]
Let $t\in [0,1]\longmapsto \varphi(t,\tau,s)$ be a curve of least action which starts from the point $\exp(\tau u)$ and ends at the point $\exp^c(v+sw)$. It follows that the cross-curvature (see Definition \ref{crosscurvature}) is given by
\begin{equation}\label{MTWsimple}
 \mathcal C(u,v,w)=-\frac{3}{2}\partial_s^2\partial_\tau^2 \int_0^1\left[\frac{1}{2}|\partial_t\varphi|^2-V(\varphi)\right]dt\Bigg|_{s=\tau=0}
\end{equation}
Since $t\longmapsto\varphi(t,\tau,s)$ is a curve of least action,
we have $D_t\partial_t\varphi=-\nabla V_\varphi$
as in Theorem \ref{potentialJacobi}. It follows that
\[
\begin{split}
\partial_\tau \left[\frac{1}{2}\left|\partial_t\varphi\right |^2+V(\varphi)\right]  &=\left<D_\tau\partial_t \varphi,\partial_t\varphi\right >+dV(\partial_\tau\varphi)\\& =\partial_t\left<\partial_t\varphi,\partial_\tau\varphi\right>+2dV(\partial_\tau\varphi).
\end{split}
\]

Since the energy $\frac{1}{2}|v|^2+V(x)$ is invariant along the curve $(x(t),\dot x(t))=(\varphi,\partial_t\varphi)$, the left side of the above equation is independent of $t$. So we integrate with respect to $t$ and get
\[
t\partial_\tau \left[\frac{1}{2}\left|\partial_t\varphi\right |^2+V(\varphi)\right] =\left<\partial_\tau\varphi,\partial_t\varphi\right>\Big|_0^t+2\int_0^t dV(\partial_\tau\varphi)dt.
\]

The minimizers $t\longmapsto \varphi(t,\tau,s)$ all end at the point $\exp^c(v+sw)$ independent of $t$. Therefore, $\partial_\tau\varphi\Big|_{t=1}=0$ and the above equation yields
\[
(t-1)\partial_\tau \left[\frac{1}{2}\left|\partial_t\varphi\right |^2+V(\varphi)\right] =\left<\partial_\tau\varphi,\partial_t\varphi\right>-2\int_t^1 dV(\partial_\tau\varphi)dt.
\]
If we set $t=0$, then the above equation becomes
\[
\partial_\tau \left[\frac{1}{2}\left|\partial_t\varphi\right |^2+V(\varphi)\right] =-\left<\partial_\tau\varphi,\partial_t\varphi\right>\Big|_{t=0}+2\int_0^1 dV(\partial_\tau\varphi)dt.
\]

Recall that $\varphi\Big|_{t=0}=\exp(\tau u)$, so $D_\tau\partial_\tau\varphi\Big|_{t=0}=0$. Therefore, if we differentiate the above equation with respect to $\tau$, then we have
\[
\begin{split}
&\partial_\tau^2 \left[\frac{1}{2}\left|\partial_t\varphi\right |^2+V(\varphi)\right]\Big|_{\tau=0} \\&= -\left<\partial_\tau\varphi,D_\tau\partial_t\varphi\right>\Big|_{\tau=t=0}+2\int_0^1 \partial_\tau^2(V(\varphi))dt\Big|_{\tau=0}\\&= -\left<u,\mathcal J^c(\sigma(s),u)\right>\Big|_{\tau=t=0}+2\int_0^1 \partial_\tau^2(V(\varphi))dt\Big|_{\tau=0}
\end{split}
\]
It follows that the cross-curvature is given by
\[
\begin{split}
\mathcal C(u,v,w)& =-\frac{3}{2}\partial_s^2\partial_\tau^2 \int_0^1\frac{1}{2}|\partial_t\varphi|^2-V(\varphi)dt\Bigg|_{s=\tau=0}\\& =-\frac{3}{2}\partial_s^2\partial_\tau^2\left[ \frac{1}{2}|\partial_t\varphi|^2+V(\varphi)-2\int_0^1 V(\varphi)dt\right]\Bigg|_{s=\tau=0}
\\&=\frac{3}{2}\partial_s^2\left<u,\mathcal J^c(\sigma(s),u)\right>\Bigg|_{s=0}
\end{split}.
\]
\end{proof}

\section{A new example: the harmonic oscillator}
\label{S:HO}

In this section, we discuss the example in Theorem \ref{eg}. More precisely, we have the following.

\begin{thm}[Cross-curvature of the harmonic oscillator vanishes]\label{egdetail}
Let $L$ be the Lagrangian defined by
\[
L(x,v)=\frac{1}{2}|v|^2-\frac{1}{2}x\cdot Ax
\]
where $A$ is a symmetric matrix satisfying $A\leq 0$. Then the corresponding cost $c$ defined by (\ref{cost}) satisfies conditions {\Azero}, {\Aone}, and {\Atwo} on $\Real^n\times \Real^n$ (see Section \ref{background} for the definitions of the conditions). The cross-curvature for the cost $c$ is identically zero. In particular, $c$ satisfies condition \CCw, a fortiori \Athreew, on $\Real^n\times \Real^n$.
\end{thm}

\begin{proof}
We first show that the corresponding cost given by (\ref{cost}) satisfies condition \Azero, \Aone, and \Atwo\ on $\Real^n\times\Real^n$. According to Theorem \ref{A012}, it is enough to check that any two points can be connected by a unique minimizers and there are no conjugate points.

The curves of least action $\gamma(\cdot)$ corresponding to the Lagrangian  $L(x,v)=\frac{1}{2}|v|^2-\frac{1}{2}x\cdot Ax$ satisfy
\[
 \partial_t^2\gamma(t)=-A\gamma(t)
\]
by Theorem \ref{potentialJacobi}.
Let $\hat e_1,...,\hat e_n$ be a basis of eigenvectors for the symmetric matrix $A$. Let
$x=\sum_{i=1}^nx^i\hat e_i$ and $v=\sum_{i=1}^nv^i\hat e_i$. If $\gamma(0)=x$ and $\dot\gamma(0)=v$, then $\gamma(\cdot)$ is given by
\[
 \gamma(t)=\sum_{i=1}^ng^i(t,x^i,v^i)\hat e_i,
\]
where $g^i(t,x^i,v^i)=\begin{cases}
                     x^i\cosh(\lambda_it)+\frac{v^i}{\lambda_i}\sinh(\lambda_it) & \hbox{if $\lambda_i\neq 0$}\\
            x^i+tv^i &\hbox{if $\lambda_i=0$}
                    \end{cases}$.

Note that the map
\[
v\longmapsto e^{1\cdot\vec H}(x,v)=\sum_{i=1}^ng^i(1,x^i,v^i)\hat e_i
\]
is a diffeomorphism from the tangent space $T_x\Real^n$ to $\Real^n$ for each $x$. Therefore,
given any two points, there is a unique path of least action joining them.
It also follows that there is no conjugate point.

Next, we show that the cross-curvature is identically zero. The Jacobi equation for this Lagrangian is given by Theorem \ref{potentialJacobi} to be
\[
\partial^2_t J+AJ=0.
\]

Note that the matrix $A$ is independent of time $t$. It follows that the Jacobi map $\mathcal J^c(y,u)$ depends only on $u$ but not on $y$. Thus Theorem \ref{cross} implies the cross-curvature is identically zero.
\end{proof}

\section{Perturbation by a Gentle Potential}

In this section, we consider the perturbed Lagrangian
\[
L_\eps(x,v)=\frac{1}{2}|v|^2-\eps V(x)
\]
defined on the trivial tangent bundle $T\Real^n$, where $V:\Real^n\longrightarrow\Real$ is a smooth function which is bounded above. We find conditions on $V$ for which the corresponding costs are positively cross-curved or satisfy the strong MTW condition for all small enough $\eps >0$.

The Hamiltonian (after identifying the tangent and cotangent bundle using the Euclidean metric) corresponding to the above Lagrangian is given by
\[
H_\eps(x,v)=\frac{1}{2}|v|^2+\eps V(x).
\]
Let $e^{t\vec H_\eps}$ be the Hamiltonian flow and let $\Phi^\eps_t$ be the map defined by
\[
\Phi^\eps_t(x,\alpha)=\pi(e^{t\vec H_\eps}(x,\alpha)).
\]

Suppose for each $\eps >0$ that $\mathcal U_\eps$ is an open subset of the tangent bundle $T\Real^n$ such that all elements in $\mathcal U_\eps$ are regular and outside the conjugate locus (see
Appendix \S \ref{appendix} for the definitions). The following theorem is a more precise version of Theorem \ref{goodperturb}.

\begin{thm}[Gentle potentials yielding positive cross-curvature]
\label{goodperturbprecise}
Let $\mathcal U$ be a bounded open subset of the tangent bundle $T\Real^n$ such that
\begin{equation}\label{goodperturbcondition}
\int_0^1\int_0^\tau\left<u,(1-t)\partial_s^2 \Hess\, V_{x+t(v+sw)}u\right>\Big|_{s=0}dtd\tau\geq C
\end{equation}
for some constant $C>0$, for all $(x,v)$ in the set $\mathcal U$ and for all unit tangent vectors $u,w$ in the tangent space $T_x\Real^n$. Then, the costs corresponding to the Lagrangians $L_\eps$ are positively cross-curved on the set \[
\mathcal K_\eps=\{(x,\Phi_1^\eps(x,\alpha))|(x,\alpha)\in\mathcal U_\eps\cap\mathcal U\}.
\]
for all small enough $\eps >0$.

If the condition (\ref{goodperturbcondition}) only holds under the assumption that $u$ and $w$ are orthogonal, then the costs satisfy the condition {\Athrees} on the set $\mathcal K_\eps$ for all small enough $\eps >0$.
\end{thm}

\begin{proof}
Let $t\in [0,1]\longmapsto\gamma_\eps(t,s)\in M$ be curves of least action corresponding to the perturbed Lagrangian $L_\eps$ which satisfy the conditions $\gamma_\eps(0,s)=0$ and $\partial_t\gamma_\eps(0,s)=v+sw$. Let $J_\eps$ be the Jacobi field along the minimizer $t\longmapsto \gamma_\eps(t,s)$ which satisfies $J_\eps\Big|_{t=0}=u$ and $J_\eps\Big|_{t=1}=0$. It follows from Theorem \ref{potentialJacobi} that
\[
\partial_t^2J_\eps+\eps \Hess \,V_{\gamma_\eps}J_\eps=0.
\]

If we differentiate the above equation with respect to $\eps$ and let $X=\partial_\eps J_\eps\Big|_{\eps=0}$, then we have
\begin{equation}\label{perturbJacobi}
\partial_t^2X+\Hess \,V_{\gamma_0}J_0=0
\end{equation}
with boundary conditions $X\Big|_{t=0}=0$ and $X\Big|_{t=1}=0$.

The family of curves $\gamma_0$ is clearly given by
\[
 \gamma_0(t,s)=x+t(v+sw)
\]
and the field $J_0$ is given by
\[
 J_0(t)=(1-t)u.
\]

Therefore, if we integrate the differential equation (\ref{perturbJacobi}) and apply the boundary conditions, we get
\begin{equation}\label{perturbJacobisoln}
 \partial_tX\Big|_{t=0}=\int_0^1\int_0^\tau (1-t)\Hess\,V_{x+t(v+sw)}ud\tau dt.
\end{equation}
Here the integral signs denote componentwise integration.

Let $Z_\eps=Z_\eps(u,v,w)$ be a smooth function such that
\[
\begin{split}
&\partial_s^2\mathcal J^{c_\eps}(\exp^{c_\eps}(v+sw),u)\Big|_{s=0}\\&=\partial_s^2\mathcal J^{c_0}(\exp^{c_0}(v+sw),u)\Big|_{s=0}+\eps \partial_s^2\partial_tX(u,v,w)\Big|_{t=s=0}+Z_\eps(u,v,w),
\end{split}
\]
where $\lim\limits_{\eps\longrightarrow 0}\frac{Z_\eps}{\eps}=0$ uniformly on the bounded set $S$ defined by
\[
S:=\{(x,u,v,w)|u,v,w\in T_xM, (x,v)\in\mathcal U, |u|=|w|=1\}.
\]

Since $\partial_s^2\left<u,\mathcal J^{c_0}(\exp^{c_0}(v+sw),u)\right>\Big|_{s=0}=0$, we have
\[
\begin{split}
&\partial_s^2\left<u,\mathcal J^{c_\eps}(\exp^{c_\eps}(v+sw),u)\right>\Big|_{s=0} \\&=\eps\left<u, \partial_s^2\partial_tX(u,v,w)\right>\Big|_{t=s=0}+\left<u,Z_\eps(u,v,w)\right>.
\end{split}
\]

It follows from (\ref{perturbJacobisoln}) and the assumptions of the theorem that
\[
\partial_s^2\left<u,\mathcal J^{c_\eps}(\exp^{c_\eps}(v+sw),u)\right>\Big|_{s=0}\geq C'>0
\]
for some constant $C'$ and for all small enough $\eps>0$.

Finally, the term
\[
\partial_s^2\left<u,\mathcal J^{c_\eps}(\exp^{c_\eps}(v+sw),u)\right>\Big|_{s=0}
\]
is homogeneous of degree two in both the $u$ and $w$ variables. Therefore,
\[
\partial_s^2\left<u,\mathcal J^{c_\eps}(\exp^{c_\eps}(v+sw),u)\right>\Big|_{s=0}\geq C'|u|^2|w|^2
\]
holds on the set
\[
\{(x,u,v,w)|u,v,w\in T_xM, (x,v)\in\mathcal U\}.
\]

Combining Theorem \ref{cross} with Definition \ref{crosscurvature} and \ref{MTWcurvature} concludes the proof.
\end{proof}

\section{Radially Symmetric Potentials}

In this section, we simplify the condition in Theorem \ref{goodperturbprecise} by assuming that the potential $V$ is radially symmetric. More precisely, we have the following.

\begin{thm}[Some gentle radial potentials which yield \Athrees\ costs]
\label{symmetric}
Setting $V=f(|x|^2/2)$, let
$\mathcal U$ be a bounded open subset of the tangent bundle $T\Real^n$ such that
\[
f''\big({\textstyle \frac{|x+tv|^2}2}\big)\geq C>0,
\quad f'''\big({\textstyle \frac{|x+tv|^2}2}\big)\geq 0,
\quad f^{(4)}\big({\textstyle \frac{|x+tv|^2}2}\big)\geq 0,
\]
for all time $t$ in the interval $[0,1]$, for some constant $C>0$, for all $(x,v)$ in the
set $\mathcal U$, and for all unit tangent vectors $u,w$ orthogonal to each other
in the tangent space $T_x\Real^n$. Then, the costs corresponding to the Lagrangians $L_\eps$ satisfy the condition {\Athrees} on the set $\mathcal K_\eps$ for all small enough $\eps >0$.
\end{thm}

\begin{proof}
Let $V(x)=f\big(\frac{|x|^2}{2}\big)$. A computation shows that the second derivatives are given by
\[
\partial_{x_i}\partial_{x_j}V=f'\big(\textstyle{\frac{|x|^2}{2}}\big)\delta_{ij}
+ f''\big({\textstyle \frac{|x|^2}{2}}\big)x_ix_j.
\]
Therefore, the Hessian in this case is given by
\[
\Hess \,V= f'\big(\textstyle{\frac{|x|^2}{2}}\big)I+ f''\big(\textstyle{\frac{|x|^2}{2}}\big)x\otimes x,
\]
where $x\otimes x$ denotes the linear transformation defined by $x\otimes x(y)=\left<x,y\right>x$.

Assume that $|u|=|w|=1$ and $\left<u,w\right>=0$. A computation shows that
\[
\begin{split}
&\left<u,\partial_s^2 \Hess\, V_{x+t(v+sw)}u\right>\Big|_{s=0}
\\&= t^2 f''\big(\textstyle{\frac{|x+tv|^2}{2}}\big)
+ [\left<x+tv,tw\right>^2 +\left<x+tv,tu\right>^2] f'''\big(\textstyle{\frac{|x+tv|^2}{2}}\big)+
\\&\quad +\left<x+tv,tw\right>^2\left<x+tv,u\right>^2 f^{(4)}\big(\textstyle{\frac{|x+tv|^2}{2}}\big)
\\&\geq t^2C.
\end{split}
\]
Therefore, Theorem \ref{goodperturbprecise} applies and the result follows.
\end{proof}

\begin{rem}
It is clear that the condition $f''\geq C$ implies that $f$ is not bounded above. It follows that there is no function which satisfies the conditions in Theorem \ref{symmetric} everywhere.
\end{rem}

\begin{ex}[The anharmonic oscillator with a quartic potential well]
Let $f$ be a function such that $f(s)=s^2$ on $[0,K]$ which stays bounded above on $[0,\infty[$.
Then it is clear that $f$ satisfies the conditions in Theorem \ref{symmetric} on the set
\[
\mathcal U=\{(x,v) \mid |x|+|v|<K\}.
\]
\end{ex}

\section{Appendix: Hamiltonian Mechanics on Manifolds}
\label{appendix}

This appendix is devoted to a discussion of background material
on Hamiltonian mechanics, and properties of the
transportation cost functions which arise by minimizing a Lagrangian action,
which supply some perspective on the results above.
Although its contents may be familiar to experts,
we include it for the readers' convenience,
since we are not aware of a suitable reference summarizing this material in the literature.

Let $L:TM\longrightarrow\Real$ be a smooth function called Lagrangian.
We define the corresponding cost functions $c_t$ by
\begin{equation}\label{Lagrange}
c_t(x,y)=\inf\int_0^tL(\gamma(s),\dot\gamma(s))ds
\end{equation}
where the infimum is taken over all smooth curves $\gamma(\cdot)$ joining $x$ to $y$ (i.e. $\gamma(0)=x$ and $\gamma(t)=y$).

The minimization problem (\ref{Lagrange}) which defined the cost functions $c_t$ above is called
Hamilton's principle of least action. In order to characterize its minimizers --- called
paths of least action --- we first consider the Legendre transform
$H:T^*M\longrightarrow\Real$ of the Lagrangian $L$ defined by
\[
H(x,\alpha)=\sup_{v\in T_xM}[\alpha(v)-L(x,v)]
\]
where the supremum is taken over all tangent vectors $v$ in the tangent space $T_xM$ at the point $x$. The function $H$ is called the Hamiltonian corresponding to the Lagrangian $L$.

Let $q^1,...,q^n,p_1,...,p_n$ be the canonical local coordinates of the cotangent bundle $T^*M$. The Hamiltonian vector field $\vec H$ of the Hamiltonian $H$ is defined via the above local coordinates by
\[
\vec H:=\left(\frac{\partial H}{\partial p_1},...,\frac{\partial H}{\partial p_n}, -\frac{\partial H}{\partial q^1},...,-\frac{\partial H}{\partial q^n}\right).
\]
The following theorem is classical (see, for instance, \cite{Fa}).

\begin{thm}[Existence of least action paths]\label{Tonelli}
Let $\left<\cdot,\cdot\right>$ be a Riemannian metric and let $|\cdot|$ be the corresponding norm. Assume that the Lagrangian $L$ satisfies the following conditions:
\begin{enumerate}
\item the restriction $L\big|_{T_xM}$ of the Lagrangian $L$ to each tangent space $T_xM$ has a positive definite Hessian,
\item $L$ satisfies $L(x,v)\geq |v|+C_1$ for all tangent vectors $v$ and for some constant $C_1$
\item  for each compact set $K\subseteq M$, and constant $C_2\geq 0$, there is a constant $C_3$ such that $L(x,v)$ satisfies $L(x,v)\geq C_2|v|+C_3$ for all tangent vector $v$ in the tangent space $T_xM$ and all $x$ in $K$.
\end{enumerate}
Then, given any time $t>0$ and any pair of points $x$ and $y$, there exists a curve of the form $s\longmapsto \gamma(s):=\pi(e^{s\vec H}(x,\alpha))$ such that $\gamma(0)=x$, $\gamma(t)=y$, and $\gamma(\cdot)$ achieves the infimum in (\ref{Lagrange}).
\end{thm}

Let us consider the optimal transportation problem with cost function given by $c_1$ defined in (\ref{Lagrange}). Under the assumptions of Theorem \ref{Tonelli} and that the initial measure $\mu$ is absolutely continuous with respect to the Lebesgue measure, it is known that there is a unique solution to the optimal transportation problem (see \cite{BeBu,FaFi}).

Let $\pi:T^*M\longrightarrow M$ be the natural projection $\pi(x,\alpha)=x$ and let $\Phi_t:T^*M\longrightarrow M$ be the map defined by
\[
\Phi_t(x,\alpha)=\pi(e^{t\vec H}(x,\alpha)).
\]
Note that the curves of least action take the form
$t\in [0,1]\longmapsto \Phi_t(x,\alpha)$ by Theorem \ref{Tonelli}.
In the Riemannian case, these extend to geodesics which are not necessarily
length minimizing and the map $\Phi_1$ is the Riemannian exponential map.

\begin{defn}[Regular covector]
A covector $(x,\alpha)$ in the cotangent space $T_x^*M$ is regular if
$t\in [0,1]\longmapsto \Phi_t(x,\alpha)$ is a unique path of least action between its endpoints.
\end{defn}

\begin{defn}[Conjugate locus]
A covector $(x,\alpha)$ is in the conjugate locus if the map $\Phi_1(x,\cdot)$ does not have full rank at $\alpha$.
\end{defn}

The conjugate locus can be characterized using Jacobi field as in Riemannian geometry. For this, let $\sigma(\cdot)$ be a curve in the cotangent bundle $T^*M$. The vector fields $J$ of the form $J(t)=\partial_s\Phi_t(\sigma(s))\Big|_{s=0}$ defined along the curve $t\longmapsto\Phi_t(\sigma(0))$ are called Jacobi fields.

\begin{thm}[Characterizing the conjugate locus with Jacobi fields]
The covector $(x,\alpha)$ is contained in the conjugate locus if and only if there is a Jacobi field defined along the curve $t\longmapsto \Phi_t(x,\alpha)$ which vanishes at the endpoints (i.e. $J|_{t=0}=0=J|_{t=1}$).
\end{thm}

\begin{proof}
The covector $(x,\alpha)$ is contained in the conjugate locus if and only if there is a tangent vector $v$ in the tangent space $T_\alpha T^*_xM$ based at the point $\alpha$ such that $\partial_\alpha\Phi_1(v)=0$. Let $s\longmapsto\sigma(s)$ be a curve in $T^*_xM$ such that $\sigma'(0)=v$, then $\Phi_t(\sigma(s))$ defines a family of minimizers and $J(t):=\partial_s\Phi_t(\sigma(s))\Big|_{s=0}$ is a Jacobi field which vanishes at the endpoints. Conversely, let $J(\cdot)$ be a Jacobi field which vanishes at the endpoints, then there is a curve $\sigma(\cdot)$ in the cotangent bundle such that $J(t)=\partial_s\Phi_t(\sigma(s))\Big|_{s=0}$. The vanishing of $J(0)$ implies that $\sigma(\cdot)$ can be chosen to be in one cotangent space $T_x^*M$. The vanishing of $J(1)$ implies that the differential of the map $\Phi_1\Big|_{T_x^*M}$ sends $\sigma'(0)$ to 0.
\end{proof}

Let $\mathcal U$ be an open subset of the cotangent bundle $T^*M$ such that all elements in $\mathcal U$ are regular and not on the conjugate locus. Let $\mathcal O=\mathcal M\times\mathcal N$ be an open subset of the product $M\times M$ contained in
\[
\{(x,\Phi_1(x,\alpha))|(x,\alpha)\in\mathcal U\}.
\]

\begin{thm}[Action minimizing transportation costs]\label{A012}
The cost $c=c_1$ defined in (\ref{Lagrange}) satisfies the conditions {\Azero}, {\Aone}, and {\Atwo} on the set $\mathcal O$.
\end{thm}

\begin{proof}
 We first prove that the function $c$ is smooth on the set $\mathcal O$. Let $(x,\alpha)$ be a point in the cotangent bundle $T^*M$ such that $t\longmapsto \Phi_t(x,\alpha)$ is uniquely minimizing and $\Phi_1(x,\alpha)=y$. Since the differential of the map $\Phi_1$ is of full rank on $\mathcal U$, there is a neighborhood $U_1\times U_2$ in the product $M\times M$ and a map $\Psi:U_1\times U_2\to T^*M$ such that $\Psi(x,y)=\alpha$, $\Psi(x',y')=x'$, and $\Phi_1(x',\Psi(x',y'))=y'$ for all $(x',y')$ in the set $U_1\times U_2$. By shrinking the set $U_1\times U_2$, we can assume that the image $\Psi(U_1\times U_2)$ is contained in $\mathcal U$. It follows that
\[
c(x',y')=\int_0^1L(\Phi_t(x',\Psi(x',y')),\partial_t\Phi_t(x',\Psi(x',y')))dt
\]
for all pairs $(x',y')$ in the set $U_1\times U_2$. It follows immediately that the cost $c$ is smooth.

For the proof of condition \Aone, consider the following minimization problem:
Find a curve which minimizes the following expression among all smooth curves $\gamma(\cdot)$ starting at $y$
\begin{equation}\label{Bolza}
 -c(\gamma(1),y) + \int_0^1L(\gamma(s),-\dot\gamma(s))ds.
\end{equation}

Clearly, the above functional is non-negative and it is zero if and only if
$s\longmapsto \gamma(1-s)$ is curve of least action connecting $\gamma(1)$ to $y$
for (\ref{Lagrange}) with Lagrangian $L$.
Let $\gamma(\cdot)$ be such a minimizer with $\gamma(0)=x$. Note that the Hamiltonian corresponding to the Lagrangian $(x,v)\longmapsto L(x,-v)$ in (\ref{Bolza}) is given by $H_r(x,\alpha):=H(x,-\alpha)$. Therefore, by \cite[Theorem 2.3]{AgLe} (see also \cite{CaSi}), there exists a covector $\alpha$ in the cotangent space $T_y^*M$ such that $\gamma(t)=\pi(e^{t\vec H_r}(\alpha))$ and $e^{1\cdot\vec H_r}(\alpha)=d_x c(\gamma(1),y)$. Finally if we let $\tilde\gamma(t)=\gamma(1-t)$, $\alpha(t)=e^{t\vec H}(\alpha)$, and $\tilde\alpha(t)=-\alpha(1-t)$, then they satisfy
\[
\pi(e^{t\vec H}(-d_xc(x,y)))=\pi(e^{t\vec H}(\tilde\alpha(0)))=\tilde\gamma(t)
\]
and $\tilde\gamma(\cdot)$ is a curve minimizing the action (\ref{Lagrange}) between its endpoints.
In particular, we have
\[
\Phi_1(-d_xc(x,y)))=\pi(e^{1\cdot\vec H}(-d_xc(x,y)))=\tilde\gamma(1)
\]

Therefore, if $(x,y)$ is contained in the set $\mathcal O$, then the above equation shows that $d_xc$ is injective as a map from $\mathcal N$ to $T^*_xM$ and $(d_xc)^{-1}(\alpha)=\Phi_1(-\alpha)$.

The condition \Atwo\ follows from the above characterization of $(d_xc)^{-1}$ and inverse function theorem.
\end{proof}

In this paper, we focus on Lagrangian arising from natural mechanical systems. More precisely, let $\left<\cdot,\cdot\right>$ be a Riemannian metric defined on the manifold $M$ and let $|\cdot|$ be the corresponding norm. Let $V:M\longrightarrow\Real$ be a smooth function on the manifold $M$ called a potential. Natural mechanical Lagrangians are Lagrangians of the form
\begin{equation}\label{potential}
L(x,v)=\frac{1}{2}|v|^2-V(x).
\end{equation}

In order to apply Theorem \ref{Tonelli}, we assume that the potential $V$ is bounded above. Let $\gamma(\cdot)$ be a curve of least action corresponding to the above Lagrangian (\ref{potential}) and let $D_t$ be the covariant derivative along the curve $\gamma(\cdot)$ with respect to the given Riemannian metric $\left<\cdot,\cdot\right>$. Let $R$ be the Riemannian curvature tensor. Then the minimizer $\gamma(\cdot)$ and the Jacobi fields $J(\cdot)$, in this case, satisfy the following equations.

\begin{thm}[Mechanical geodesics and Jacobi equations]
\label{potentialJacobi}
The curves of least action $\gamma(\cdot)$ corresponding to natural mechanical Lagrangian defined in (\ref{potential}) satisfy the Newton's second law
\[
D_t\partial_t\gamma=-\nabla V_\gamma.
\]
A Jacobi field $J(\cdot)$ along a least action curve $\gamma(\cdot)$ satisfies the following equation:
\[
D_t^2J+\Hess\,V_{\gamma}(J)+R(\dot\gamma,J)\dot\gamma=0.
\]
\end{thm}

\begin{proof}
Let $\gamma_\eps(\cdot)$ be a family of smooth curves such that $\gamma_0(\cdot)$ is a curve of least action, $\gamma_\eps(0)=x$, and $\gamma_\eps(1)=y$. Then
\[
\begin{split}
0&=\frac{d}{d\eps}\int_0^1\frac{1}{2}|\partial_t\gamma_\eps(t)|^2-V(\gamma_\eps(t))dt\Big|_{\eps=0} \\&=\int_0^1\left<\partial_t\gamma_0(t),D_\eps\partial_t\gamma_\eps(t)\Big|_{\eps=0}\right> -\left<\nabla V_{\gamma_0(t)},\partial_\eps\gamma_\eps\Big|_{\eps=0}\right>dt \\&=\int_0^1\left<-D_t\partial_t\gamma_0(t)-\nabla V_{\gamma_0(t)},\partial_\eps\gamma_\eps(t)\Big|_{\eps=0}\right>dt.
\end{split}
\]
It follows that $D_t\partial_t\gamma=-\nabla V_\gamma$.

Let $s\in \Real$ parameterize a family of least action curves
$t\longmapsto\varphi(t,s)$ %be a family of curves of least action
such that $\varphi(t,0)=\gamma(t)$. Let $J=\partial_s\varphi\Big|_{s=0}$. Then,
\[
\begin{split}
D_t^2\partial_s\varphi&= D_tD_s\partial_t\varphi \\&= D_sD_t\partial_t\varphi-R(\partial_t\varphi,\partial_s\varphi)\partial_t\varphi \\&= -D_s\nabla V_\varphi-R(\partial_t\varphi,\partial_s\varphi)\partial_t\varphi \\&  =-\Hess\, V_\varphi(\partial_s\varphi)-R(\partial_t\varphi,\partial_s\varphi)\partial_t\varphi.
\end{split}
\]
After setting $s$ to $0$, we get the equation $D_t^2J+\Hess\,V_{\gamma}(J)+R(\dot\gamma,J)\dot\gamma=0$ as claimed.
\end{proof}

\end{document}